\documentclass[a4paper,12pt]{article}
\usepackage{times, url}
\usepackage{amsmath,amssymb,amsthm,amsfonts,bm,float,cite}
\usepackage{url}

\usepackage{breakurl}
\usepackage[breaklinks]{hyperref}
\newtheorem{definition}{Definition}[section]
\newtheorem{remark}{Remark}[section]

\newtheorem{lemma}[definition]{Lemma}
\newtheorem{corollary}[definition]{Corollary}

\usepackage[none]{hyphenat}
\usepackage[center]{caption}

\newtheorem{innercustomthm}{Theorem}
\newenvironment{customthm}[1]
  {\renewcommand\theinnercustomthm{#1}\innercustomthm}
  {\endinnercustomthm}

\newtheorem{innercustomcrl}{Corollary}
\newenvironment{customcrl}[1]
{\renewcommand\theinnercustomcrl{#1}\innercustomcrl}
{\endinnercustomcrl}

\newtheorem{innercustomdef}{Definition}
\newenvironment{customdef}[1]
{\renewcommand\theinnercustomdef{#1}\innercustomdef}
{\endinnercustomdef}

\begin{document}
\setcounter{page}{1}

\begin{center}
{\LARGE \bf  Number of stable digits of any integer tetration} %
\vspace{8mm}

{\Large \bf Marco Ripà$^1$ and Luca Onnis$^2$}
\vspace{3mm}

$^1$ sPIqr Society, World Intelligence Network \\ 
Rome, Italy \\
e-mail: \url{marcokrt1984@yahoo.it}
\vspace{2mm}

$^2$ Independent researcher \\ 
Cagliari, Italy \\
e-mail: \url{luca.onnis02@gmail.com}
\vspace{2mm}

\end{center}
\vspace{10mm}

\noindent
{\bf Abstract:}  In the present paper we provide a formula that allows to compute the number of stable digits of any integer tetration base $a\in\mathbb{N}_0$. The number of stable digits, at the given height of the power tower, indicates how many of the last digits of the (generic) tetration are frozen. Our formula is exact for every tetration base which is not coprime to $10$, although a maximum gap equal to $V(a)+1$ digits (where $V(a)$ denotes the constant congruence speed of $a$) can occur,  in the worst-case scenario, between the upper and lower bound. In addition, for every $a>1$ which is not a multiple of $10$, we show that $V(a)$ corresponds to the $2$-adic or $5$-adic valuation of $a-1$ or $a+1$, or even to the $5$-adic order of $a^{2}+1$, depending on the congruence class of $a$ modulo $20$. 
\vspace{5mm}

\section{Introduction} \label{sec:Intr}

The aim of this paper is to give a general formula which returns the number of stable digits \cite{3,7,13} of the tetration
\[
^{b}a = \begin{cases} a \mbox{ }\mbox{ }\mbox{ }\mbox{ }\mbox{ }\mbox{ }\mbox{ }\mbox{ }\mbox{ \hspace{1.2mm} if $b = 1$} \\
a^{(^{(b-1)}a)}  \mbox{ if $b\geq 2$}
\end{cases}\hspace{-3.5mm},
\]
for any $a\in\mathbb{N}_0$, at any given height $b\in\mathbb{N}-\{0\}$ \cite{4,5}. Let us consider the standard decimal numeral system (radix-$10$). Thus, we are interested in an easy way to find the value of $n\in\mathbb{N}_0$ such that $^{b}a\equiv{^{b+1}a}\pmod{10^{n}} \wedge ^{b}a \not\equiv{^{b+1}a}\pmod{10^{n+1}}$.

In order to simplify the notations, let us invoke the definition of the congruence speed of $^{b}a$ from Reference \cite{12}, and then (Definition \ref{def1.3}) we will extend it to the base $a = 0$.

\begin{definition} \label{def1.1}
Let $n\in\mathbb{N}_0$ and assume that $a\in\mathbb{N}-\{0,1\}$ is not a multiple of $10$. Then, given $^{b-1}a\equiv{^{b}a}\pmod {10^{n}} \wedge ^{b-1}a \not\equiv{^{b}a}\pmod {10^{n+1}}$, $\forall b>a$, $V(a,b)$ returns the strictly positive integer such that $^{b}a\equiv{^{b+1}a}\pmod {10^{n+V(a,b)}}$ $\wedge$ $^{b}a \not\equiv{^{b+1}a}\pmod {10^{n+V(a,b)+1}}$, and we define $V(a,b)$ as the “congruence speed” of the base $a$ at the given height of its hyperexponent \linebreak $b\in\mathbb{N}-\{0\}$.
\end{definition}

Consequently, if $a = 2$, the tetrations for $b$ from $1$ to $5$ are $^12 = 2$, $^22 = 4$, $^32 =16$, $^42 = 65536$, and $^52 = \dots 19156736$ (respectively), so we can see that $V(2,1) = V(2,2) = 0$, whereas $V(2,3) = V(2,4) = 1$. \\
From \cite{11,12} we know that, for any given $a$ which is not a multiple of $10$, exists a unique “optimal” value, $\Bar{b}:=\min_{b}\{b\in \mathbb{N}-\{0\}: V(a,b) = V(a,b+k),\forall k\in \mathbb{N}_0\}$, of the hyperexponent which guarantees $V(a,\Bar{b}+k) = V(a)$ for any $k\in \mathbb{N}_0$ \cite{7}, reaching a height of $a+1$ represents a sufficient but not necessary condition for the constancy of the congruence speed, since \linebreak $V(a,a+1) = V(a)$ is always true. Improved bounds for $\Bar{b}(a)$ will be introduced in the next section.

\begin{definition} \label{def1.2}
Let the tetration base $a \in \mathbb{N}-\{0,1\}$ not be a multiple of $10$, and then let $\bar{b}:=\min_{b}\left\{b \in \mathbb{N}-\{0\}: V(a, b)=V(a, b+k), \forall k \in \mathbb{N}_{0}\right\}$. We define as ``constant congruence speed" of $a$ the non-negative integer $V(a):=V(a, \bar{b})$.
\end{definition}

\begin{definition} \label{def1.3}
Let $a=1$, then $V(1,1)=1$ and $V(1, b)=0=V(1)$ for any $b \geq 2$. We also define $V(0)=0$ for any $b \in \mathbb{N}-\{0\}$, since it is possible to extend the domain of tetration by considering that $\lim _{a \rightarrow 0}$$^{b}a:=$$^{b}0$ implies ${ }^{b} 0=1$ if $b$ is even and ${ }^{b} 0=0$ otherwise (see \cite{2}). Thus, for any $b \geq 1$, $^{b}0$ does not produce any stable digit, and $V(0, b)=V(0)=0$ by Definition \ref{def1.1}.
\end{definition}
Since, in general, $n$ depends on $a$ and $b$ (see Definition \ref{def1.1}), from here on, let us denote by $\# S_{c}(a, b)$ the number of stable digits of all the bases belonging to the congruence class \linebreak $c \pmod {10}$ (e.g., if we consider only tetration bases which have $3$ or $7$ as their rightmost digit, we will indicate the number of their stable digits, at height $b$, by $\# S_{\{3,7\}}(a, b)$).

For any given pair $(a, b)$ of positive integers, and assuming that $c \in\{1,2,3,4,5,6,7,8,9\}$, by definition, we have that
\begin{equation}\label{eq1}
    \# S_{c}(a, b):=\sum_{i=1}^{b} V(a, i)=\left\{\begin{array}{lll}
\sum_{i=1}^{b} V(a, i) & \text { if } & b<\bar{b} \\
\sum_{j=1}^{\bar{b}-1} V(a, j)+(b-\bar{b}+1) \cdot V(a) & \text { if } & b \geq \bar{b}
\end{array}\right..
\end{equation}

Now, in the rest of the present paper, let us assume that $a \in \mathbb{N}: a \not\equiv 0\pmod {10}$ does not belong to the congruence class $0$ modulo $10$, since, for any $b \geq 1$, if $a \equiv 0\pmod {10}$, then the number of stable digits of $^b{a}$ corresponds to $0$ if and only if $a=0$ (by Definitions \ref{def1.1} and \ref{def1.3}), and to the number of trailing zeros which appear at the end of ${ }^{b}((k+1) \cdot 10)$ otherwise (e.g., if $k=1$ and $b=2$, we have $^b{a}=^2\hspace{-1mm}{20}=2^{20} \cdot 10^{20}$ so that $\# S_{0}(20,2)=20$).

Section \ref{sec:2} describes how to calculate $V(a)$ given $a$, and consequently $\# S_{c}(a, b)$ at height $b$. \linebreak In Subsection \ref{sec:sub2.1} we present a formula that returns the exact value of $\# S_{c}(a, b)$ for any $c$ which is not coprime to $10$, whereas Subsection \ref{sec:sub2.2} is devoted to study the four remaining cases. Section \ref{sec:3} explains how to find which is the smallest hyperexponent $b$ such that $^b a$ returns any desired value of $\# S_{c}(a, b)$, for the chosen base $a$.

\section{A formula for the number of stable digits of $\linebreak$\boldmath$^b a:a \not\equiv 0 \pmod {10}$} \label{sec:2}

In this section we study $\# S_{c}(a, b)$ assuming that the last digit of the tetration base is not equal to zero so that the residues modulo $10$ of $c$ cover the whole set $\{1,2,3,4,5,6,7,8,9\}$.

For this purpose, given a prime number $p$, let us indicate the $p$-adic order on $\mathbb{Z}$ by \linebreak $v_{p}: \mathbb{Z} \rightarrow \mathbb{N} \cup \{+\infty\}$. $v_{p}$ is a valuation on $\mathbb{Z}$ (since $\{0\} \subset \mathbb{Z} \subset \mathbb{Q}$ and the statement follows by Theorem of Reference \cite{8}) and, given $d \in \mathbb{Z}$, it is defined as the mapping
$$
v_{p}(d):=\left\{\begin{array}{cc}
\max \left\{q \in \mathbb{N}_{0}: p^{q} \mid d\right\} & \text { if } d \neq 0 \\
+\infty & \text { if } d=0
\end{array}\right.\hspace{-1mm},
$$
(i.e., $v_{p}(d)$ is the function which returns the highest exponent $q$ such that $p^{q}$ divides $d$, so we write $v_{3}(18)=2$ since $p=3$ is a prime, $3^{2} \mid 18$, and $3^{3} \nmid 18$ ) \cite{9}.

Assuming $r \in \mathbb{Z}$, the $p$-adic valuation is characterized by some interesting properties \cite{1,8}, such as $v_{p}(d \cdot r)=v_{p}(d)+v_{p}(r), v_{p}(d+r) \geq \min \left\{v_{p}(d), v_{p}(r)\right\}$ (e.g., given any prime $p$, $\min \left\{v_{p}(d-1), v_{p}(d+1)\right\} \leq v_{p}(2 \cdot d)$ holds for any $d \hspace{0.3mm}$). Moreover, if $v_{p}(d) \neq v_{p}(r)$, then $v_{p}(d+r)=\min \left\{v_{p}(d), v_{p}(r)\right\}$ and, in particular, $v_{p}(d)<v_{p}(r) \Rightarrow v_{p}(d+r)=v_{p}(d)$.

Now, from \cite{12}, we know that the constant congruence speed of any given base $a$ which is not congruent to $0$ modulo $5$ is (always) less than or equal to the $5$-adic valuation of
\begin{itemize}
\vspace{-1mm}
    \item $a-1$ if $a \equiv 1 \hspace{-1mm}\pmod 5$;
    \vspace{-1mm}
    \item $a^{2}+1$ if $a \equiv\{2,3\}\hspace{-1mm}\pmod 5$;
    \vspace{-1mm}
    \item $a+1$ if $a \equiv 4 \hspace{-1mm}\pmod 5$;
\end{itemize}
\vspace{-1mm}
while, if $a: a \equiv 5 \hspace{-1mm}\pmod{10}$, we have that $V(a)+1=v_{2}\left(a^{2}-1\right)$ (see \cite{12}, Corollary 2.2, \linebreak pp. 55-56).

\begin{lemma}\label{lem2.1}
If the tetration base $a$ belongs to the congruence class $5$ modulo $20$, then \linebreak $V(a)=v_{2}(a-1)$. If $a$ belongs to the congruence class $15$ modulo $20$ , then $V(a)=v_{2}(a+1)$.
\end{lemma}
\begin{proof}
 Let $a$ be such that $a \equiv 5\pmod {10}$. Since $a=10 \cdot k+5$ is an odd integer for any $k \in \mathbb{N}_{0}$, it follows that the argument of $v_{2}(a \pm 1)$ is even, and by definition we have \linebreak $a \pm 1=2^{n} \cdot h_{1}$, for some $n$ and $h_{1} \in \mathbb{N}-\{0\}$. Now, let us consider Equation (26) from \cite{12}, assume $n \geq 2$ and observe how, for any $m \in \mathbb{N}_{0}$, the parity of $\grave{h_{1}}:=\grave{h_{1}}(n)$ does not change from $\frac{\left(2^{n} \cdot\left((-1)^{n-1}+2\right)-i^{n \cdot(n-1)}\right)+m \cdot 10 \cdot 2^{n} \pm 1}{2^{n}}=\grave{h_1}$ to $\frac{\left(2^{n} \cdot\left((-1)^{n-1}+2\right)-i^{n \cdot(n-1)}\right) \pm 1}{2^{n}}=h_{1}$ so that $v_{2}(a-1)=n \Rightarrow\left(h_{1}=1\right.$ iff $n: n \equiv 2\pmod{4} \wedge h_{1}=3$ iff $\left.n: n \equiv 3\pmod{4}\right)$,
while 
$v_{2}(a\hspace{0.3mm}+\hspace{0.3mm}1)\hspace{0.7mm}=\hspace{0.8mm}n \hspace{0.8mm}\Rightarrow\hspace{0.8mm}\left(h_{1}\hspace{0.2mm}=\hspace{0.1mm}1 \hspace{0.7mm}\text { iff } \hspace{0.7mm}n\hspace{0.5mm} : \hspace{0.8mm}n\hspace{0.7mm} \equiv \hspace{0.8mm}0 \pmod{4} \hspace{0.45mm}\wedge \hspace{0.45mm}h_{1}\hspace{0.6mm}=\hspace{0.6mm}3 \hspace{0.5mm}\text { iff } \hspace{0.8mm}n : n \equiv 1 \pmod{4}\right)$.

Similarly, considering any $m \in \mathbb{N}_{0}$, the other half of the cases are covered by
$$
\frac{\left(2^{n} \cdot\left((-1)^{n}+8\right)+i^{n \cdot(n-1)}\right)+m \cdot 10 \cdot 2^{n} \pm 1}{2^{n}}=\grave{h_{1}} \Rightarrow \frac{\left(2^{n} \cdot\left((-1)^{n}+8\right)+i^{n \cdot(n-1)}\right) \pm 1}{2^{n}}=h_{1},
$$
and it follows that, for any $n \in \mathbb{N}-\{0,1\}, h_{1}=7$ or $h_{1}=9$.

Consequently, $\grave{h_{1}}$ is always an odd number (i.e., $h_{1} \in\{1,3,7,9\}$), and $a: a \equiv 5\pmod{10} \Rightarrow$ $V(a)=v_{2}(a-1) \vee v_{2}(a+1)$.
Now, $\max \left\{v_{2}(a-1), v_{2}(a+1)\right\}=\max \{1, V(a)\}=V(a)$ (since, for any given $k \in \mathbb{N}_{0}$ such that $\frac{a-1}{2}=\frac{10 \cdot k+4}{2}$, if $\frac{10 \cdot k+4}{2}$ is odd, then $\frac{10 \cdot k+6}{2}$ is even, and vice versa). Thus, $a: a \equiv 5 \pmod{10} \Rightarrow\left(v_{2}(a-1)=1 \wedge v_{2}(a+1)=V(a)\right) \vee\left(v_{2}(a-1)\right.$ $\left.=V(a) \wedge v_{2}(a+1)=1\right)$. Hence (see \cite{9}, Definition \ref{def2.1}), $v_{2}(a-1)+v_{2}(a+1)=V(a)+1$ $\Rightarrow V(a)+1=v_{2}((a-1) \cdot(a+1)) \Rightarrow V(a)+1=v_{2}\left(a^{2}-1\right)$.

\sloppy Thus, we have proved that if $a$ belongs to the congruence class $5$ modulo $10$, then \linebreak
$V(a)=\max \left\{v_{2}(a-1), v_{2}(a+1)\right\}$ $=v_{2}\left(a^{2}-1\right)-1$.

In order to conclude the proof, we need to show that $V(a)=v_{2}(a-1)$ if and only if $k$ is even. Although this result could be easily achieved by observing that if $\frac{a-1}{2}=5 \cdot k+2$, then $2^{2} \mid(a-1)$ if and only if $k$ is even so that $\left(v_{2}(a-1) \geq 2\right) \wedge\left(v_{2}(a+1)=1\right)$ and (since $V(a) \geq 2$ for any $a: a \equiv 5\pmod{10}$\cite{11,12}) $k: k \equiv 0\pmod{2} \Rightarrow v_{2}(a-1)=V(a)$, we take this opportunity to extend the basic technique that will be used for proving Theorem \ref{t2.1}. 

For this purpose, let $h_{2} \in \mathbb{N}-\{0\}$ and $h_{1} \in \mathbb{N}-\{0\}$ (as usual) so that $(10 \cdot k+5)^{2}-1$ $=2^{n+1} \cdot h_{2}$ and $10 \cdot k+5-1=2^{n} \cdot h_{1}$ (or equivalently, $10 \cdot k+5+1=2^{n} \cdot h_{1}$). Since we have already verified that $h_{1}$ is odd, we need to find for which values of $k$ we get an odd value of $h_{2}$.

Thus,
$$
10 \cdot k+5-1=2^{n} \cdot h_{2} \Rightarrow \frac{h_{2}}{h_{1}}=\frac{(10 \cdot k+5)^{2}-1}{2 \cdot(10 \cdot k+5-1)} \Rightarrow h_{2}=(5 \cdot k+3) \cdot h_{1}
$$
and it follows that $h_{2}$ is odd if and only if $k$ is even (whereas
$$
10 \cdot k+5+1=2^{n} \cdot h_{2} \Rightarrow \frac{h_{2}}{h_{1}}=\frac{(10 \cdot k+5)^{2}-1}{2 \cdot(10 \cdot k+5+1)} \Rightarrow h_{2}=(5 \cdot k+2) \cdot h_{1}
$$
so that $h_{2}$ is odd if and only if $k$ is odd).

Hence, $k \equiv 0\hspace{-1mm}\pmod{2} \Rightarrow V(10 \cdot k+5)=$ $v_{2}(10 \cdot k+4)$ and $k \equiv 1\hspace{-1mm}\pmod{2} \Rightarrow V(10 \cdot k+5)=v_{2}(10 \cdot k+6)$.

Therefore, $a=20 \cdot k+5 \Rightarrow V(a)=v_{2}(a-1)$ and $a=20 \cdot k+15 \Rightarrow V(a)=v_{2}(a+1)$. This completes the proof.
\end{proof}
Moreover, it is possible to invert Equations (\ref{eq6}), (\ref{eq7}), (\ref{eq10}), (\ref{eq11}), (\ref{eq14})-(\ref{eq17}) from Reference \cite{12} in order to simplify the computation of the exact value of $V(a)$ given $a$, extending Lemma \ref{lem2.1} from $a: a \equiv 5\hspace{-1mm}\pmod{10}$ to any tetration base which is not a multiple of $10$.
\begin{customthm}{2.1}\label{t2.1}
For any $a \in \mathbb{N}_{0}$ such that $a$ is not a multiple of $10$, the constant congruence speed of $a$ is given by Equation (\ref{eq2}),

\begin{equation}\label{eq2}
V(a)=\left\{\begin{aligned}
0 & \textnormal  { if } a \in\{0,1\} \\
\min \left\{v_{2}(a-1), v_{5}(a-1)\right\} & \textnormal  { if } a \equiv 1\hspace{-3.5mm}\pmod{100} \wedge a \neq 1 \\
\min \left\{v_{2}(a+1), v_{5}(a-1)\right\} & \textnormal  { if } a \equiv 51\hspace{-3.5mm}\pmod{100}\\
v_{5}\hspace{-1mm}\left(a^{2}+1\right) & \textnormal  { if } a \equiv\{2,8\}\hspace{-3.5mm}\pmod{10}\\
\min \left\{v_{2}(a+1), v_{5}\hspace{-1mm}\left(a^{2}+1\right)\right\} & \textnormal  { if } a \equiv\{7,43\}\hspace{-3.5mm}\pmod{100} \\
\min \left\{v_{2}(a-1), v_{5}\hspace{-1mm}\left(a^{2}+1\right)\right\} & \text { if } a \equiv\{57,93\}\hspace{-3.5mm}\pmod{100} \\
v_{5}(a+1) & \textnormal  { if } a \equiv 4\hspace{-3.5mm}\pmod {10}\\
v_{2}\hspace{-1mm}\left(a^{2}-1\right)-1 & \text { if } a \equiv 5\hspace{-3.5mm}\pmod{10} \\
v_{5}(a-1) & \textnormal  { if } a \equiv 6\hspace{-3.5mm}\pmod {10} \\
\min \left\{v_{2}(a-1), v_{5}(a+1)\right\} & \textnormal { if } a \equiv 49\hspace{-3.5mm}\pmod{100} \\
\min \left\{v_{2}(a+1), v_{5}(a+1)\right\} & \textnormal  { if } a \equiv 99\hspace{-3.5mm}\pmod{100} \\
1 & \textnormal  { otherwise }
\end{aligned}\right..
\end{equation}
\end{customthm}
\begin{proof}
Although $a \hspace{-0.3mm} \in\hspace{-0.3mm}\{0,1\}\hspace{-0.3mm} \Rightarrow \hspace{-0.3mm}V(a)=0$ by Definition \ref{def1.1}, Equation (18) of Reference \cite{12} gives us a sufficient condition (relative to the congruence class of $a$ modulo $25$) for $V(a)=1$. The set of all the tetration bases that are congruent to $\{2,3,4,6,8,9,11,12,13,14,16$, $17,19,21,22,23\}\pmod{25}$ contains the residual values of $a$ which satisfy the last line of (\ref{eq2}). It follows that Equation (\ref{eq2}) is true for every $a$ such that $V(a) \leq 1$. Thus, we assume $a: V(a) \geq 2$ in the rest of this proof.

Theorem \ref{t2.1} can be proved by combining Equations (6), (7), (10), (11), (14)-(17), (26) from Reference \cite{12} and Equation (\ref{eq2}) above so that it is possible to extend the method anticipated in the proof of Lemma \ref{lem2.1} (which we will explicitly illustrate here for the symmetrical cases $a \equiv 4\pmod{10}$ and $a \equiv 6\pmod{10}$ to all the congruence classes of $a$ modulo $10$ (or $100$) considered by (\ref{eq2}).

First of all, we show that Theorem \ref{t2.1} is true for any tetration base belonging to the congruence class $4$ modulo $10$. Thus, let $a=4+10 \cdot k$, for any $k \in \mathbb{N}_{0}$. Since Equation (\ref{eq10}) from Reference \cite{12} (which maps all the bases congruent to $4$ modulo $10$ that are characterized by a constant congruence speed of $n$) has already been proven to be true, we only need to check that, $\forall m \in \mathbb{N}_{0}$, $V(a)=v_{5}(a+1) \Rightarrow 5^{n}-1+m \cdot 2 \cdot 5^{n}+1=5^{n} \cdot h$ if and only if $m \equiv\{0,1,3,4\}\pmod{5}$ and $h$ is not divisible by $5$. We immediately see that $h$ is a multiple of $5$ if and only if $h$ belongs to the congruence class $5$ modulo $10$, since $h \equiv 0\pmod{10}$ $\Rightarrow 5^{n} \cdot h \equiv 0\pmod{10}$, but this would lead to a contradiction because $5^{n} \cdot h \equiv 0\pmod{10}$ implies that $5^{n} \cdot h$ cannot be written as $5^{n}-1+m \cdot 2 \cdot 5^{n}+1$, for the reason that $m \cdot 2 \cdot 5^{n} \equiv 0\pmod{10}$, and so $5^{n}+m \cdot 2 \cdot 5^{n} \equiv 5\pmod{10}$. Thus, for any $n \in \mathbb{N}-\{0,1\}$, $m: m \equiv\{0,1,3,4\}\pmod{5}$ represents a necessary and sufficient condition for $5^{n} \cdot(2 \cdot m+1)=5^{n} \cdot h$, since we have previously verified that $5^{n}-1+m \cdot 2 \cdot 5^{n}+1=5^{n} \cdot h$ holds if and only if $h \not \equiv 5\pmod{10}$. Hence, $2 \cdot m+1=h \hspace{1mm} \left(\forall m \in \mathbb{N}_{0}: m \not \equiv 2\pmod{5}\right)$. Similarly, by observing that $5^{n}+1+m \cdot 2 \cdot 5^{n}-1=5^{n} \cdot h \Rightarrow 2 \cdot m+1=h$, we can verify that $m \not \equiv 2\pmod{5}$ is also a necessary and sufficient condition for $V(a: a \equiv 6\pmod{10})$ $=v_{5}(a-1)$.

In order to complete the proof of Theorem \ref{t2.1}, we note that it is possible to adopt the same approach as above for all the congruence classes listed in (\ref{eq2}). Although $\exists^{\infty} a: \tilde{v}(a)>V(a)$, we can check that $a \equiv\{2,4,5,6,8\}\pmod{10} \Rightarrow \nexists a: \tilde{v}(a) \neq V(a)$ and (as a direct consequence) this reduces the non-strict inequality
$$
V\left(a_{\{2,8\}}(n)\right)=n \leq \hat{n}=\frac{\ln\hspace{-0.8mm} \left(\frac{{a_{\{2,8\}}\hspace{0mm}^{2}}\hspace{0.3mm}+\hspace{0.3mm}1}{h_{\{2,8\}}}\right)}{\ln (5)},
$$
stated by Equations (\ref{eq12}), (\ref{eq13}) from Reference \cite{12}, to an identity between $n$ and $\hat{n}$ (let $5 \nmid h$ and consider that, for any $k \in \mathbb{N}_{0}\hspace{0.3mm}$, $a^{2}+1=5^{n} \cdot h \Rightarrow 5^{n}=\frac{5 \cdot\left(20 \cdot k^{2}+8 \cdot k+1\right)}{h}$ if $a=10 \cdot k+2$ and $5^{n}=\frac{5 \cdot\left(20 \cdot k^{2}+32 \cdot k+13\right)}{h}$ if $a=10 \cdot k+8$, taking also into account Equations (\ref{eq16}), (\ref{eq17}) in \cite{12}).

In order to prove (\ref{eq2}), line 8 (from top to bottom), by using the method presented above instead of Lemma \ref{lem2.1}, an interesting computational exercise would be to autonomously verify that $a_{5}(n)$ from Equation (26) of Reference \cite{12} is equal to $\sqrt{2^{n+1} \cdot h+1}$ only if $h$ is odd, independently confirming that $v_{2}\left(a^{2}-1\right)-1=V(a)$ for any $a: a \equiv 5\hspace{-1mm}\pmod{10}$.

The next part of this proof is focused on the special subset $\mathcal{E}=\{1,51,43,93,7,57,49,99\}$ of the set of all the bases that are coprime to $10$. It is easy to verify from \cite{11,12} that if $a$ modulo $100$ does not belong to $\mathcal{E}$, then $V(a)=\tilde{v}(a)$ (while $a\hspace{-1mm}\pmod{100} \in \mathcal{E} \not\Rightarrow V(a) \neq \tilde{v}(a)$). Since $\exists^{\infty} a: \tilde{v}(a)>V(a)$ for each of the eight congruence classes of $a$ modulo $100$ which correspond to the eight elements of the set $\{1, 51, 43, 93, 7, 57, 49, 99\}$, we anticipate that the same idea (already described in the previous paragraphs) can be used to prove the given result for all the aforementioned congruence classes (in order to easily find the exact value of the constant congruence speed of any $a: a\hspace{-1mm}\pmod{100} \in \mathcal{E}$). For brevity, let us show how to apply the usual method with reference only to the congruence class $1$ of $a$ modulo $100$. Thus, if $\tilde{v}(a)=V(a)$ is true for any tetration base $a \in \mathbb{N}-\{1\}$ such that $a \equiv 1 \hspace{-1mm}\pmod{100}$, by merging Equation (6), Reference \cite{12}, line 2, and the first line of (\ref{eq2}) from the present Theorem \ref{t2.1}, we should be able to prove that $V(a)=v_{5}(a-1) \Rightarrow 10^{n} \cdot(m+1)+1-1=5^{n} \cdot h \quad(\forall n \in \mathbb{N}-\{0,1\})$ if and only if $m \in \mathbb{N}_{0}: m \not \equiv 9\hspace{-1mm}\pmod{10}$ (since the constraint $h: h \not \equiv 5\hspace{-1mm}\pmod{10}$ follows from the $5$-adic valuation definition). This is not possible to do, since $h \not \equiv 5\hspace{-1mm}\pmod{10}$ holds for any $m \not \equiv 4\hspace{-1mm}\pmod{5}$ too (e.g., if $n=2$ and $m=4$, then $V(501)=2 \neq 3=v_{5}(a-1)$, whereas $V(501)=v_{2}(a-1)$). Consequently, $V(a: a=1+100 \cdot k) \neq v_{5}(a-1)$ if and only if the first digit to the left of the rightmost (trailing) zero(s) is a five, while $V(1+100 \cdot k) \neq v_{2}(a-1)$ can occur only if that “key digit”  (see Definition \ref{def2.3} assuming $\alpha_{x_{2} x_{1}}=\alpha_{01}$ by Definition \ref{def2.2}) is even, and the conclusion follows since the sets of odd and even numbers have no intersection.

This technique is enough to confirm that $V(a)=\min \left\{v_{2}(a-1), v_{5}(a-1)\right\}$ for any $a \neq 1: a \equiv 1\hspace{-1mm}\pmod{100}$, and similarly showing that $V(100 \cdot k-1) \neq v_{5}(a+1)$ and $V(100 \cdot k-1)=v_{2}(a+1)$ if and only if the first digit to the left of the rightmost repunit ($9$'s) is a four (see Definitions \ref{def2.2} and \ref{def2.3} assuming $\alpha_{x_{2} x_{1}}=\alpha_{99}$).

Therefore, the statement of Theorem \ref{t2.1} can be proved by performing all the basic (tedious) calculations which arise when we merge Equations (6), (7), (10), (11), (14)-(18), (26) from Reference \cite{12} and Equation (\ref{eq2}) of the present paper.
\end{proof}
\begin{customcrl}{2.1}\label{cor2.1}
If a is congruent to $25$ modulo $100$, then $V(a)=v_{2}(a-1)$. If $a$ is congruent to $75$ modulo $100$, then $V(a)=v_{2}(a+1)$.
\end{customcrl}
\begin{proof}
We observe that $a \equiv 25\pmod{100} \Rightarrow a \equiv 5\pmod{20}$ and $a \equiv 75\pmod{100} \Rightarrow$ $a \equiv 15\pmod{20}$. Then, we invoke Lemma \ref{lem2.1} once again, and this concludes the proof of Corollary \ref{cor2.1}.
\end{proof}
\begin{corollary}\label{cor2.2}
For any $a \in \mathbb{N}-\{1\}$ such that $a \not \equiv 0\hspace{-0.5mm}\pmod{10}$,
\begin{equation} \label{eq3}
V(a) = \begin{cases}\mbox{ }\mbox{ }\min \left\{v_{2}(a-1),\right. \left.v_{5}(a-1)\right\} \mbox{ \textnormal{if} } a \equiv 1\hspace{-3mm}\pmod{20} \\
\mbox{ }\mbox{ }\min \left\{v_{2}(a+1),\right.\left.v_{5}(a-1)\right\}  \mbox{ \textnormal{if} } a \equiv 11\hspace{-3mm}\pmod{20} \\
\mbox{ }\mbox{ }\mbox{ }\mbox{ }\mbox{ }\mbox{ }\mbox{ }\mbox{ }\mbox{ }\mbox{ }\mbox{ }\mbox{ }\mbox{ }\mbox{ }\mbox{ }\mbox{ }\mbox{ }\mbox{ }\mbox{ }\mbox{ }\mbox{ }\mbox{ }\mbox{ }\mbox{ }\mbox{ }\mbox{ }\mbox{ }\mbox{ }\mbox{ }\mbox{ }v_{5}\left(a^{2}+1\right) \mbox{ \textnormal{if} } a \equiv\{2,8\}\hspace{-3mm}\pmod{10} \\
\min \left\{v_{2}(a+1),\right. \left.v_{5}\left(a^{2}+1\right)\right\} \textnormal { if } a \equiv\{3,7\}\hspace{-3mm}\pmod{20} \\
\min \left\{v_{2}(a-1),\right. \left.v_{5}\left(a^{2}+1\right)\right\} \textnormal { if } a \equiv\{13,17\}\hspace{-3mm}\pmod{20} \\
\mbox{ }\mbox{ }\mbox{ }\mbox{ }\mbox{ }\mbox{ }\mbox{ }\mbox{ }\mbox{ }\mbox{ }\mbox{ }\mbox{ }\mbox{ }\mbox{ }\mbox{ }\mbox{ }\mbox{ }\mbox{ }\mbox{ }\mbox{ }\mbox{ }\mbox{ }\mbox{ }\mbox{ }\mbox{ }\mbox{ }\mbox{ }\mbox{ }\mbox{ }\mbox{ }\mbox{ }\mbox{ }\mbox{ }v_{5}(a+1) \textnormal { if } a \equiv 4\hspace{-3mm}\pmod{10} \\
\mbox{ }\mbox{ }\mbox{ }\mbox{ }\mbox{ }\mbox{ }\mbox{ }\mbox{ }\mbox{ }\mbox{ }\mbox{ }\mbox{ }\mbox{ }\mbox{ }\mbox{ }\mbox{ }\mbox{ }\mbox{ }\mbox{ }\mbox{ }\mbox{ }\mbox{ }\mbox{ }\mbox{ }\mbox{ }\mbox{ }\mbox{ }\mbox{ }\mbox{ }\mbox{ }\mbox{ }\mbox{ }\mbox{ }v_{2}(a-1)  \textnormal { if } a \equiv 5\hspace{-3mm}\pmod{20} \\
\mbox{ }\mbox{ }\mbox{ }\mbox{ }\mbox{ }\mbox{ }\mbox{ }\mbox{ }\mbox{ }\mbox{ }\mbox{ }\mbox{ }\mbox{ }\mbox{ }\mbox{ }\mbox{ }\mbox{ }\mbox{ }\mbox{ }\mbox{ }\mbox{ }\mbox{ }\mbox{ }\mbox{ }\mbox{ }\mbox{ }\mbox{ }\mbox{ }\mbox{ }\mbox{ }\mbox{ }\mbox{ }\mbox{ }v_{2}(a+1) \textnormal { if } a \equiv 15\hspace{-3mm}\pmod{20} \\
\mbox{ }\mbox{ }\mbox{ }\mbox{ }\mbox{ }\mbox{ }\mbox{ }\mbox{ }\mbox{ }\mbox{ }\mbox{ }\mbox{ }\mbox{ }\mbox{ }\mbox{ }\mbox{ }\mbox{ }\mbox{ }\mbox{ }\mbox{ }\mbox{ }\mbox{ }\mbox{ }\mbox{ }\mbox{ }\mbox{ }\mbox{ }\mbox{ }\mbox{ }\mbox{ }\mbox{ }\mbox{ }\mbox{ }v_{5}(a-1) \textnormal { if } a \equiv 6\hspace{-3mm}\pmod{10} \\
\mbox{ }\mbox{ }\min \left\{v_{2}(a-1),\right. \left.v_{5}(a+1)\right\} \textnormal { if } a \equiv 9\hspace{-3mm}\pmod{20} \\
\mbox{ }\mbox{ }\min \left\{v_{2}(a+1),\right. \left.v_{5}(a+1)\right\} \textnormal { if } a \equiv 19\hspace{-3mm}\pmod{20}
\end{cases}\hspace{-2mm}.
\end{equation}
\end{corollary} 
\begin{proof}
Trivially,
$$
\begin{gathered}
a \equiv 1\hspace{-3mm}\pmod{100} \Rightarrow a \equiv 1\hspace{-3mm}\pmod{20},\\
a \equiv 51\hspace{-3mm}\pmod{100} \Rightarrow a \equiv 11\hspace{-3mm}\pmod{20},\\
a \equiv 43\hspace{-3mm}\pmod{100} \Rightarrow a \equiv 3\hspace{-3mm}\pmod{20},\\
a \equiv 93\hspace{-3mm}\pmod{100} \Rightarrow a \equiv 13\hspace{-3mm}\pmod{20},\\
a \equiv 7\hspace{-3mm}\pmod{100} \Rightarrow a \equiv 7\hspace{-3mm}\pmod{20},\\
a \equiv 57\hspace{-3mm}\pmod{100} \Rightarrow a \equiv 17\hspace{-3mm}\pmod{20},\\
a \equiv 49\hspace{-3mm}\pmod{100} \Rightarrow a \equiv 9\pmod{20},\\
a \equiv 99\hspace{-3mm}\pmod{100} \Rightarrow a \equiv 19\hspace{-3mm}\pmod{20}.
\end{gathered}
$$

Incorporating the statement of Lemma \ref{lem2.1}, all the allowed congruence classes of $a$ modulo $20$ have been covered so that there is no way and no need to specify when $V(a)$ is unitary.

This concludes the proof of Corollary \ref{cor2.2}.
\end{proof}
For the sake of simplicity, Definition \ref{def2.1} introduces a compact notation for a general (tight) upper bound of the exact value of $V(a)$, as it follows from Equations (\ref{eq2})\&(\ref{eq3}).

\begin{customdef}{2.1}\label{def2.1}
Let $a \neq 1$ be such that $a \not \equiv 0\hspace{-1mm}\pmod{10}$. We define
$$
\tilde{v}(a):= \begin{cases}v_{5}(a-1) & \textnormal { iff } a \equiv 1\hspace{-3mm}\pmod{5} \\ v_{5}\left(a^{2}+1\right) & \textnormal { iff } a \equiv\{2,3\}\hspace{-3mm}\pmod{5} \\ v_{5}(a+1) & \textnormal { iff } a \equiv 4\hspace{-3mm}\pmod{5} \\ v_{2}\left(a^{2}-1\right)-1 & \textnormal { iff } a \equiv 5\hspace{-3mm}\pmod{10} \end{cases}
$$
so that $\tilde{v}(a) \geq V(a)$ for any $a$.
\end{customdef}

\subsection{The exact value of \boldmath$\#S_{\{2,4,5,6,8\}}(a,b)$}
\label{sec:sub2.1}

Assuming radix-$10$ \cite{3}, as usual, we describe the structure $\# S(a, b)$ by providing an exact formula for any pair $(a, b)$ such that $a \equiv\{2,4,5,6,8\}\pmod {10} \wedge b \geq 3$, and very tight bounds which hold for all the bases $a: a \equiv\{1,3,7,9\}\pmod {10}$. In particular, if $a \equiv\{2,4,5,6,8\}\pmod{10}$, then $V(a)=\tilde{v}(a)$ (see Theorem \ref{t2.1} and Definition \ref{def2.1}).

Let $k \in \mathbb{N}_{0}$ and assume that $a=(20 \cdot k+2 \vee 20 \cdot k+18)$. Then, for any $a: a \equiv\{2,18\}\pmod{20}$, $^{1}a \equiv\{2,8\}\pmod{10}$, $^{2}a \equiv 4\pmod{10}$, and finally $^{3}a \equiv\hspace{0mm} ^{4}a \pmod{10}$ since $^{3}a \equiv 6\pmod{10}$. It follows that
\begin{equation}\label{eq4}
\# S_{\{2,8\}}(20 \cdot k+2 \vee 20 \cdot k+18, b)=\left\{\begin{array}{ccc}0 & \text { if } & b=1 \\ (b-2) \cdot V(a) & \text { if } & b \geq 2\end{array}\right.
\end{equation}
\[
=\left\{\begin{array}{ccc}0 & \text { if } \quad b=1 \\ (b-2) \cdot v_{5}\left(a^{2}+1\right) & \text { if } \quad b \geq 2\end{array} .\right.
\]

If $a: a \equiv\{12,8\}\pmod{20}$, then $\forall b \in \mathbb{N}-\{0\}$
\begin{equation}\label{eq5}
    \# S_{\{2,8\}}(20 \cdot k+12 \vee 20 \cdot k+8, b)=(b-1) \cdot V(a)=(b-1) \cdot v_{5}\left(a^{2}+1\right) .
\end{equation}

Even if the cases $a: a \equiv 4\pmod {10}$ and $a: a \equiv 6\pmod {10}$ have already been fully described in References \cite{10,12}, “repetita iuvant”, and so (for any $b$) we have
\begin{equation}\label{eq6}
  \# S_{4}(a, b)=(b-1) \cdot V(a)=(b-1) \cdot v_{5}(a+1) ;  
\end{equation}
while $a: a \equiv 6\pmod {10}$ trivially implies $V(a, 1) \geq 1 \Rightarrow V(a, b) \geq 1$ so that (for any $b$)
\begin{equation} \label{eq7}
 \# S_{6}(a, b \geq 2)=(b+1) \cdot V(a)=(b+1) \cdot v_{5}(a-1)   
\end{equation}
immediately $\hspace{2mm}$follows $\hspace{2mm}$from $\hspace{2mm} V(a \equiv 6\pmod {10}, \hspace{3.5mm} 1)+V(a \equiv 6\pmod {10}, \hspace{3.5mm} 2)=$\linebreak$3 \cdot V(a \equiv 6\pmod {10}, b \geq 3)=3 \cdot v_{5}(a-1)$.

If $a: a \equiv 5\pmod {10}$, then $V(a)=v_{2}\left(a^{2}-1\right)-1$ (see Lemma \ref{lem2.1}), and $\bar{b}(a)$ is always equal to $3$, with the only exception of the base $a=5$ (i.e., $\bar{b}(5)=4 \neq 3=\bar{b}(10 \cdot k+15)$, $\forall k \in \mathbb{N}_{0}$). It follows that

\begin{equation} \label{eq8}
    \# S_{5}(20 \cdot k+15, b \geq 2)=b \cdot\left(v_{2}\left(a^{2}-1\right)-1\right)+1;
\end{equation}
\begin{equation} \label{eq9}
    \# S_{5}(20 \cdot k+25, b \geq 2)=(b+1) \cdot\left(v_{2}\left(a^{2}-1\right)-1\right)\hspace{-0.7mm};
\end{equation}
\begin{equation} \label{eq10}
    \# S_{5}(5, b)= \begin{cases}1 & \text { iff } \hspace{2mm} b=1 \\
    4 & \text { iff } \hspace{2mm} b=2 \\
8+2 \cdot(b-3) & \text { iff } \hspace{2mm} b \geq 3\end{cases} .
\end{equation}

In order to complete the $\# S(a, b)$ map, we need to study all the tetration bases which are coprime to $10$, and this will be the goal of the next subsection.

\subsection{Bounding \boldmath$\#S_{\{1,3,7,9\}}(a,b)$ from \boldmath$V(a)$} \label{sec:sub2.2}
Let $a: a \not\equiv 0\pmod{10} \hspace{0.3mm}\wedge \hspace{0.3mm} a \neq 1$ be given (bearing in mind that $V(1,1)=1$, whereas $V(1, \bar{b}(1))=V(1,2)=V(1)=0$, and also $V(0, \bar{b}(0))=V(0,1)=V(0)=0$) so that $V(a)$ is fully described by Theorem \ref{t2.1}, and $V(a) \leq \tilde{v}(a)$ always holds (Definition \ref{def2.1}).

Under the above-mentioned condition $a \neq 1$, we note that if $V(a: \operatorname{gcd}(a, 10)=1, b)=0$, then $a=(20 \cdot k+3 \vee 20 \cdot k+7) \wedge b=1$, for any $k \in \mathbb{N}_0$. 

Thus,
\vspace{-1mm}
\begin{equation}\label{eq11}
\begin{aligned}
\# S_{\{3,7\}}(a=(20 \cdot k+3 \vee 20 \cdot k+7), b \geq \bar{b}(a)-1)
=&\left\{\begin{array}{c}
(b-1) \cdot V(a) \text { iff } V(a, 2)=V(a) \\
b \cdot V(a)+1 \hspace{3mm} \text { iff } V(a, 2)>V(a)
\end{array}\right.
\end{aligned}
\end{equation}
(e.g.,
$$
\begin{gathered}
V(6907922943,2)=11>9=v_5\left(6907922943^2+1\right) \\
\Rightarrow \# S_3(a=(20 \cdot 345396147+3), b \geq \bar{b}(a))=\# S_3(6907922943, b \geq 6) \\
=b \cdot V(a)+1 \hspace{0.3mm}\textnormal{,}
\end{gathered}
$$
while
$$
\begin{gathered}
V(107,2)=2=v_5\left(107^2+1\right) \\
\Rightarrow S_7(a=(20 \cdot 5+7), b \geq \bar{b}(a)-1)=\# S_7(107, b \geq 1) \\
=((b-1) \cdot V(a))\hspace{0.5mm}\textnormal{)}\hspace{0.3mm}.
\end{gathered}
$$

For any $b$, the above also implies the bound (\ref{eq12})
\begin{equation}\label{eq12}
\begin{aligned}
(b-1) \cdot V(a) & \leq \# S_{\{3,7\}}(a=(20 \cdot k+3 \vee 20 \cdot k+7), b) \\
& \leq b \cdot V(a)+1
\end{aligned}
\end{equation}
and the (weaker) relation (\ref{eq13}) follows
\begin{equation}\label{eq13}
\begin{aligned}
(b-1) \cdot\left(v_5\left(a^2+1\right)\right) & \leq \# S_{\{3,7\}}(a=(20 \cdot k+3 \vee 20 \cdot k+7), b \geq 2) \\
& \leq b \cdot\left(v_5\left(a^2+1\right)\right)\hspace{-0.7mm}.
\end{aligned}
\end{equation}

Finally, for any $a \equiv\{1, 3, 7, 9\}\hspace{-1mm}\pmod{10}$ which cannot be written as $20 \cdot k+3 \vee 20 \cdot k+7$, the number of stable digits of $^{b}a$ at height $b \geq \bar{b}(a)-1$ is $b \cdot V(a)$, or $b \cdot V(a)+1$, or $(b+1) \cdot V(a)$.

We can also derive the following general bound which holds for any $b \geq 2$,
\begin{equation}\label{eq14}
\begin{aligned}
b \cdot V(a) & \leq \# S_{\{1,3,7,9\}}(a \neq(20 \cdot k+3 \vee 20 \cdot k+7), b \geq 2) \\
& \leq(b+1) \cdot V(a)\hspace{0.3mm}\textnormal{,}
\end{aligned}
\end{equation}
and we additionally state that $\bar{b}(a) \leq v_5\left(a^2+1\right)+2$ is valid for every tetration base $a$ which is congruent to $\{3,7\}\pmod{10}$. The aforementioned limit on $\bar{b}(a)$ arises by combining the upper bounds by Equations (\ref{eq12})\&(\ref{eq14}) with the general constraint from Equation (\ref{eq15}) (see Section \ref{sec:3}), taking also into account that if $a \not \equiv\{0, 2, 8\} \pmod{10}$, then $V(a, 2)$ always assumes a strictly positive value.

Furthermore, if $a \not \equiv\{3,7\}\hspace{-1mm}\pmod{20}$, then $\bar{b}(a) \leq \tilde{v}(a)+1$, since we have not to worry about the case $V(a, 1)=0$, which cannot happen (the only $a$ which is characterized by $V(a, 2)>0$ and such that $\bar{b}(a)>\tilde{v}(a)+1$ is the base $5$, but we already know that
$\bar{b}(5)=\tilde{v}(5)+2$).\linebreak
In general, assuming $a \neq 5$, only a maximum of $\tilde{v}(a)$ additional iterations can occur from the first time that the congruence speed assumes a strictly positive value (i.e., the first step or the second one for any $a$ which is coprime to $10$) to the last time that $V(a, b)>V(a)$. Thus, for any $a$ which is not congruent to $0$ modulo $10$, the maximum theoretical value of $\bar{b}(a)$ is bounded above by $1+\tilde{v}(a)+1$.

Therefore, $\bar{b}(a) \leq \tilde{v}(a)+2$ for every $a: a \not\equiv 0\pmod{10}$ (let us observe that $a=1 \Rightarrow$ $\tilde{v}(1)=v_5(0)=\infty$ and $\bar{b}(1)=2$ by definition), and this result confirms also Conjecture 1 of Reference \cite{12}.

We can take a look at the congruence speed of the base $a=163574218751$ as a random check on the upper bound provided by (\ref{eq14}). \hspace{0.5mm}$a=163574218751$ is characterized by $\tilde{v}(163574218751)=v_5(163574218751-1)=13=V(163574218751)$, so we have $V(a, 1)=12$, $V(a, 2)=19$, $V(a, 3)=V(a, 4)=V(a, 5)=V(a, 6)=15$, and $V(a, b \geq 7)=$ $V(a)=13$. Hence, by Equation (\ref{eq1}),
$$
\begin{aligned}
\# S_1(163574218751, b \geq \bar{b}) &=12+19+15 \cdot 4+(b-(\bar{b}-1)) \cdot 13 \\
&=91+(b-\bar{b}+1) \cdot 13 \\
&=(6+1) \cdot 13+(b-6) \cdot 13 \\
&=(b+1) \cdot V(163574218751) .
\end{aligned}
$$

In addition, some more bases from each one of the four critical congruence classes modulo $10$, whose $\# S_{\{1,3,7,9\}}(a, b \geq \bar{b}(a))$ is uniquely given by $(b-1) \cdot V(a)$, or $b \cdot V(a)$, or $b \cdot V(a)+1$, or $(b+1) \cdot V(a)$, are shown below:
\begin{itemize}
    \item $\# S_1(74218751, b \geq 3)=b \cdot V(a)+1=b \cdot 8+1$,
        \vspace{-1.5mm}
    \item $\# S_1(45215487480163574218751, b \geq 13)=(b+1) \cdot V(a)=(b+1) \cdot 25$;
\end{itemize}
\begin{itemize}
    \item $\# S_3(143, b \geq 2)=(b-1) \cdot V(a)=(b-1) \cdot 2$,
        \vspace{-1.5mm}
    \item $\# S_3(133, b \geq 1)=b \cdot V(a)=b$,
        \vspace{-1.5mm}
    \item $\# S_3(847288609443, b \geq 5)=b \cdot V(a)+1=b \cdot 2+1$,
        \vspace{-1.5mm}
    \item $\# S_3(2996418333704193, b \geq 17)=(b+1) \cdot V(a)=(b+1) \cdot 16$;
        \vspace{0.5mm}
\end{itemize}
\begin{itemize}
    \item $\# S_7(907, b \geq 2)=(b-1) \cdot V(a)=(b-1) \cdot 2$,
    \vspace{-1.5mm}
    \item $\# S_7(177, b \geq 1)=b \cdot V(a)=b$,
        \vspace{-1.5mm}
    \item $\# S_7(807, b \geq 6)=b \cdot V(a)+1=b \cdot 3+1$,
        \vspace{-1.5mm}
    \item $\# S_7(23418092077057, b \geq 15)=(b+1) \cdot V(a)=(b+1) \cdot 14$;
            \vspace{0.5mm}
\end{itemize}
\begin{itemize}
    \item $\# S_9(599, b \geq 1)=b \cdot V(a)=b \cdot 2$,
        \vspace{-1.5mm}
    \item $\# S_9(499, b \geq 2)=b \cdot V(a)+1=b \cdot 2+1$,
        \vspace{-1.5mm}
    \item $\# S_9(781249, b \geq 4)=(b+1) \cdot V(a)=(b+1) \cdot 6$.
\end{itemize}
\begin{remark}
The laws that let us predict the value of every $\# S_{\{1,3,7,9\}}\left(a, b \geq \bar{b}(a)\right)$ (including all the examples above) can be derived from Reference \cite{12}, Equation (\ref{eq2}) for $i=1,3,4,9, 10,12$.
\end{remark}
\begin{customdef}{2.2}\label{def2.2}
Let $x_1$ and $x_2$ belong to the set $\{0,1,2,3,4,5,6,7,8,9\}$. Let $n \in \mathbb{N}-\{0\}$. We denote by $\alpha_{x_2 x_1}[n] \in\{0,1,2,3,4,5,6,7,8,9\}$ the n-th rightmost digit of the unique decadic integer $\alpha_{x_2 x_1}$ satisfying the fundamental equation $y^5=y$ (stated in Reference \cite{12}, Proposition 6, pp. 47-48), such that $\alpha_{x_2 x_1}[1]=x_1$ and $\alpha_{x_2 x_1}[2]=x_2$ (e.g., if $x_2=5$ and $x_1=1$, then $\alpha_{x_2 x_1}[2]=\alpha_{51}[2]=5$ by definition and $\alpha_{x_2 x_1}[3]=\alpha_{51}[3]=7$ by construction, since $\alpha_{x_2 x_1}=1-2 \cdot\left\{2^{5^n}\right\}_{\infty}=\ldots 87480163574218751$ is the only solution of $y^5=y$ in the ring of decadic integers such that $x_2=5$ and $x_1=1$ \cite{6}). For any given pair $\left(x_2, x_1\right)$, we indicate by $\alpha_{x_2 x_1} \pmod{10^n} \in \frac{\mathbb{Z}}{10^n \mathbb{Z}}$ the (decimal) integer formed by the $n$ rightmost digits of $\alpha_{x_2 x_1} \in \mathbb{Z}_{10}$ (e.g., if $n=4$ and $\left(x_2, x_1\right)=(9,9)$, then the selected solution of $y^5=y$ in the ring $\mathbb{Z}_{10}$ is $-1=\ldots 99999$, and we have $\alpha_{x_2 x_1} \pmod{10^n}=\alpha_{99} \pmod{10^4}=9999$).
\end{customdef}
For sake of clarity, the only decadic integers which satisfy the mentioned fundamental equation, $y^5=y$ (and, in general, $y^t=y$ for any $t \in \mathbb{Z}: t \geq 5$), are
$\alpha_{00}=0=\ldots 00000000000000000000$, \\
$\alpha_{01}=1=\ldots 0000000000000000000000000000000000000000000000000000000000000000001$, \\
$\alpha_{51}=1-2 \cdot\left\{5^{2^n}\right\}_{\infty}=$ \\
$\ldots 0219875666980838272377998885153153538207781991786760045215487480163574218751$,\\
$\alpha_{32}=\left\{2^{5^n}\right\}_{\infty}=$  \\
$\ldots 0275906862593839649523223304553032451441224165530407839804103263499879186432$, \\
$\alpha_{93}=\left\{5^{2^n}\right\}_{\infty}-\left\{2^{5^n}\right\}_{\infty}=$ \\
$\ldots 9614155303915741214287777252870390779454884838576212137588152996418333704193$, \\
$\alpha_{43}=-\left\{5^{2^n}\right\}_{\infty}-\left\{2^{5^n}\right\}_{\infty}=$ \\
$\ldots 9834030970896579486665776138023544317662666830362972182803640476581907922943$, \\
$\alpha_{24}=\left\{5^{2^n}\right\}_{\infty}-1=$ \\
$\ldots 9890062166509580863811000557423423230896109004106619977392256259918212890624$, \\
$\alpha_{25}=\left\{5^{2^n}\right\}_{\infty}=$ \\
$\ldots 9890062166509580863811000557423423230896109004106619977392256259918212890625$, \\
$\alpha_{75}=-\left\{5^{2^n}\right\}_{\infty}=$\\ $\ldots 0109937833490419136188999442576576769103890995893380022607743740081787109375$, \\
$\alpha_{76}=1-\left\{5^{2^n}\right\}_{\infty}=$\\
$\ldots 0109937833490419136188999442576576769103890995893380022607743740081787109376$,\\
$\alpha_{07}=-\left\{5^{2^n}\right\}_{\infty}+\left\{2^{5^n}\right\}_{\infty}=$\\
$\ldots 0385844696084258785712222747129609220545115161423787862411847003581666295807$,\\
$\alpha_{57}=\left\{5^{2^n}\right\}_{\infty}+\left\{2^{5^n}\right\}_{\infty}=$\\
$\ldots 0165969029103420513334223861976455682337333169637027817196359523418092077057$,\\
$\alpha_{68}=-\left\{2^{5^n}\right\}_{\infty}=$\\
$\ldots 9724093137406160350476776695446967548558775834469592160195896736500120813568$,\\
$\alpha_{49}=2 \cdot\left\{5^{2^n}\right\}_{\infty}-1=$\\
$\ldots 9780124333019161727622001114846846461792218008213239954784512519836425781249$, and
$\alpha_{99}=-1= \ldots 99999999999999999999999999999999999999999999999999999999999999$.
\begin{customdef}{2.3}\label{def2.3}
Let $n \in \mathbb{N}-\{0\}$ and $l \in \mathbb{N}-\{0,1\}$. Let the tetration base $10^{n-1}<a<10^n$ be such that $a:=\sum_{j=1}^n s_j \cdot 10^j$, for $s_1 \in\{1,3,7,9\}, s_{1<j<n} \in\{0,1,2,3,4,5,6,7,8,9\}$, and (if $n$ is not unitary) $s_n \in\{1,2,3,4,5,6,7,8,9\}$. Given one pair $\left(x_2, x_1\right)$ as specified by Definition \ref{def2.2}, if $a \neq \alpha_{x_2 x_1}\pmod{10^n}$, then $l \leq n$, and we single out $s_l \in\{0,1,2,3,4,5,6,7,8,9\}$, \linebreak the l-th rightmost digit of the given tetration base, as the key-digit of \linebreak $a=$ $s_{n-} s_{(n-1)-} \dots_{-} s_{l-} s_{(l-1)-} \dots_{-}{s_{2-}}x_1$ such that $a \equiv \alpha_{x_2 x_1}\pmod{10^{l-1}} \hspace{1mm} \wedge$ \linebreak $a \not\equiv \alpha_{x_2 x_1}\pmod{10^l}$ (i.e., $\forall \bar{\jmath} \in\{1,2, \ldots, l\hspace{-1.5mm}-\hspace{-1.5mm}1\}, \quad\left(s_1=x_1 \hspace{1.5mm}\wedge\hspace{1.5mm} x_2=x_2\left(x_1\right)\right) \Rightarrow$ \linebreak $\left(\alpha_{x_2 x_1}[\bar{\jmath}]=s_j \wedge \alpha_{x_2 x_1}[l] \neq s_l\right)$, see Definition \ref{def2.2}).
If $a=\alpha_{x_2 x_1}\pmod{10^n}$, then $l>n$ and $s_l:=0$, since we shall assume $s_{(n+w)}=0$ for any $w \in \mathbb{N}-\{0\}$ (e.g., $a=57=0057$ implies $l=4$ because $\alpha_{57}\pmod{10^4} = 7057$, and so we have $s_l-\alpha_{57}[l]=\frac{a-\alpha_{57}\hspace{-1mm}
\pmod{10^l}}{10^{(l-1)}}=s_4-\alpha_{57}[4]=-7$).
\end{customdef}
When we take into account only $\tilde{v}(a)$ trying to guess the exact value of $V(a)$ by Definition \ref{def2.1}, the most obvious critical bases are originated by those digits of $\alpha_{51}$, $\alpha_{43}$, $\alpha_{93}$, $\alpha_{07}$, $\alpha_{57}$, and $\alpha_{49}$ which are equal to $5$ (e.g., $\forall l \in \mathbb{N}: \alpha_{43}[l]=5$, $\tilde{v}\left(\alpha_{43}\pmod {10^{l-1}}\right)>V\left(\alpha_{43}\pmod{10^{l-1}}\right)$). Thus, let us select one of the aforementioned decadic integers, that we will indicate by $\alpha_{\overline{x_2 x_1}}$, and perform a surgical$\pmod{10^n}$ cut on that string, just at the right of a casual digit $5$ (i.e., given one pair $\left(x_2, x_1\right) \in\{(0,7),(4,3),(4,9),(5,1),(5,7),(9,3)\}$,$\hspace{2mm}\alpha_{\overline{x_2 x_1}}[n+1]=5$ shall be satisfied for $\alpha_{x_2 x_1}=\alpha_{\overline{x_2 x_1}}\hspace{1mm}$) so that the decimal integer $\alpha_{\overline{x_2 x_1}}\pmod{10^n}$ we get is a pretty special tetration base characterized by $\tilde{v}\left(\alpha_{\overline{x_2 x_1}}\pmod{10^n}\right)>V\left(\alpha_{\overline{x_2 x_1}}\pmod{10^n}\right)$, as long as $\alpha_{\overline{x_2 x_1}}[1] \neq 7$. At this point, it is important to point out that the tetration base $7$ is very peculiar (i.e., $\tilde{v}(7)=V(7)$ despite the trueness of $\alpha_{07}[2]=5$, see (\ref{eq15})\&(\ref{eq16})) because $x_2=5$ if and only if $x_1=7 \vee x_1=1$, but $V(1)=0$ by definition, so we can see that the choice of the base 7 implies $\alpha_{07}[l]=\alpha_{07}[3]=8 \neq 5$, instead of $\alpha_{57}[l]=\alpha_{57}[2]=5$, since $7 \equiv 7\pmod{20}$ and $7 \not\equiv 17\pmod{20}$ share the common rule $v_5\left(a^2+1\right)=V(a)$ as stated by (\ref{eq16}), lines 15 and 17 (from top to bottom). Furthermore, inside the proof of Theorem \ref{t2.1}, for any $n \in \mathbb{N}-\{0,1\}$ and $l \in \mathbb{N}-\{0,1,2\}$, we have already spoiled that $a:=\alpha_{99}\pmod{10^n}$ violates the $V(a)=\tilde{v}(a)$ (wrong) conjecture \cite{7} if and only if $\alpha_{99}[l]=4$ (see Corollary \ref{cor2.3} and Reference \cite{12} for further details).

To be fair, as stated in Reference \cite{12}, Proposition 6, p. 47, there is one last fundamental intersection that arises from the solution $y_{15}(t)=1$ of $y^t=y$ over the commutative ring of decadic integers, considering the corresponding decimal integers modulo $10^n$ (by the well-known ring homomorphism).

For this purpose, as a clarifying example, let us show how $y_{15}(5): 1^5=1$ works (see \cite{12}, pp. 47-48). Given $a(n):=\sum_{j=1}^n s_j \cdot 10^j$, assume that $n \in \mathbb{N}-\{0,1\}$ and $s_{j=1}=1$, let $s_{1<j<n} \in\{0,1,2,3,4,5,6,7,8,9\}$ and $s_n \in\{1,2,3,4,5,6,7,8,9\}$ be defined by Definition \ref{def2.3}. When $s_2=0$, for any given set $\left\{s_3, s_4, \ldots, s_{n-1}, s_n\right\}$ as specified above, we can verify that $\# S_1(a(n), b \geq 2)=(b+1) \cdot V(a(n))$ is always true, whereas, if $s_2=s_n$ is an arbitrary element of the set $\{1,2,3,4,5,6,7,8,9\}$, then $\# S_1(a(2), b \geq 2)=(b+1) \cdot V(a(2))$ $=b \cdot V(a(2))+1$ if and only if $s_2 \neq 5$ (where $5=\alpha_{51}[2]$ by Equation (2) from Reference \cite{12}). Since $V(51,1)=2, V(51,2)=3$, and $V(51,3)=V(51)=2$, it follows that $\# S_1(51, b \geq 2)=b \cdot V(51)+1$ is not equal to $(b+1) \cdot V(51)$ (i.e., $V(a) \neq 1 \Rightarrow$ $b \cdot V(a)+1 \neq(b+1) \cdot V(a)$).

In view of the fact that there are many other exceptions to the $\tilde{v}(a)=V(a)$ rule, let us introduce the explanatory Corollary \ref{cor2.3} which reveals, in details, the general law (involving all the tetration bases which are coprime to $10$) at the bottom of the universal inequality $V(a) \leq \tilde{v}(a)$.
\begin{customcrl}{2.3}\label{cor2.3}
Let $n \in \mathbb{N}-\{0,1\}$ and let $l \in \mathbb{N}-\{0,1,2\}$ be such that $s_l$, the $l$-th rightmost digit of the tetration base $a:=\sum_{j=1}^n s_j \cdot 10^j$, is outlined by Definition \ref{def2.3}. For any $a$ coprime to $10$, the constant congruence speed is given by
\begin{equation}\label{eq15}
V(a)=\left\{\begin{aligned}
v_5(a-1) & \textnormal { if }\left(s_2, s_1\right)=(0,1) \wedge s_l \neq 5 \\
v_2(a-1) & \textnormal { if }\left(s_2, s_1\right)=(0,1) \wedge s_l=5 \\
v_5(a-1) & \textnormal { if }\left(s_2, s_1\right)=(5,1) \wedge\left|s_l-\alpha_{51}[l]\right| \neq 5 \\
v_2(a+1) & \textnormal { if }\left(s_2, s_1\right)=(5,1) \wedge\left|s_l-\alpha_{51}[l]\right|=5 \\
v_5\left(a^2+1\right) & \textnormal { if }\left(s_2, s_1\right)=(4,3) \wedge\left|s_l-\alpha_{43}[l]\right| \neq 5 \\
v_2(a+1) & \textnormal { if }\left(s_2, s_1\right)=(4,3) \wedge\left|s_l-\alpha_{43}[l]\right|=5 \\
v_5\left(a^2+1\right) & \textnormal { if }\left(s_2, s_1\right)=(9,3) \wedge\left|s_l-\alpha_{93}[l]\right| \neq 5 \\
v_2(a-1) & \textnormal { if }\left(s_2, s_1\right)=(9,3) \wedge\left|s_l-\alpha_{93}[l]\right|=5 \\
v_5\left(a^2+1\right) & \textnormal { if }\left(s_2, s_1\right)=(0,7) \wedge\left|s_l-\alpha_{07}[l]\right| \neq 5 \\
v_2(a+1) & \textnormal { if }\left(s_2, s_1\right)=(0,7) \wedge\left|s_l-\alpha_{07}[l]\right|=5 \\
v_5\left(a^2+1\right) & \textnormal { if }\left(s_2, s_1\right)=(5,7) \wedge\left|s_l-\alpha_{57}[l]\right| \neq 5 \\
v_2(a-1) & \textnormal { if }\left(s_2, s_1\right)=(5,7) \wedge\left|s_l-\alpha_{57}[l]\right|=5 \\
v_5(a+1) & \textnormal { if }\left(s_2, s_1\right)=(4,9) \wedge\left|s_l-\alpha_{49}[l]\right| \neq 5 \\
v_2(a-1) & \textnormal { if }\left(s_2, s_1\right)=(4,9) \wedge\left|s_l-\alpha_{49}[l]\right|=5 \\
v_5(a+1) & \textnormal { if }\left(s_2, s_1\right)=(9,9) \wedge s_l \neq 4 \\
v_2(a+1) & \textnormal { if }\left(s_2, s_1\right)=(9,9) \wedge s_l=4 \\
1 & \textnormal { if }\left(s_2, s_1\right) \notin\{(0,1),(0,7),(4,3),(4,9),(5,1),(5,7),(9,3),(9,9)\}
\end{aligned}\right..
\end{equation}
\end{customcrl}
\begin{proof}The last line of (\ref{eq15}) trivially follows from (18) of Reference \cite{12}. $\hspace{-0.7mm}$In order to prove the main result of Corollary \ref{cor2.3}, we can verify that the constraints given by (\ref{eq15}) on the pairs $\left(s_2, s_1\right)$ represent sufficient conditions for the trueness of the stated $2$-adic / $5$-adic valuation rules. \linebreak
For this reason, let $n,(l-1) \in \mathbb{N}-\{0,1\}$ be as specified by the corollary itself. We note that Equation (2) of \cite{12}, (by construction) implies that, $\forall\left(s_2, s_1\right) \in\{(0,1),(0,7),(4,3),(4,9),$ $(5,1),(5,7),(9,3),(9,9)\}$, $\left|s_l-\alpha_{x_2 x_1}[l]\right|=5 \Rightarrow V(a)+1 \leq \tilde{v}(a)$ and from the proof of Theorem \ref{t2.1} we immediately deduce that $\left(\left(s_{\bar{\jmath}}=9, \forall \bar{\jmath} \in\{1,2, \ldots, l-1\}\right) \wedge s_l=4\right) \Rightarrow$ $\tilde{v}(a) \geq V(a)+1$ (at this purpose, we point out that $\left|s_l-\alpha_{99}[l]\right|=|4-9|=5$ for any $l \in \mathbb{N}-\{0,1,2\}$). We can repeat the process for all the given $\left(s_2, s_1\right)$ pairs, confirming the aforementioned relation.

Now, let $\breve{v}_2(a):=v_2(a-1)$ if $a \equiv\{1,49,57,93\}\pmod{100}$ and $\breve{v}_2(a):=v_2(a+1)$ if $a \equiv\{7, 43, 51, 99\}\pmod{100}$. It follows that $\left|s_l-\alpha_{x_2 x_1}[l]\right| \equiv 0\pmod{2} \Rightarrow V(a)+1$ $\leq \breve{v}_2(a)$ always holds, whereas $\left|s_l-\alpha_{x_2 x_1}[l]\right| \equiv 1\pmod{2} \Rightarrow V(a)=\breve{v}_2(a)$ is generally true (the direct verification can be accomplished through the standard technique, performing simple calculations, as described in the proof of Lemma \ref{lem2.1}).

Basically, given $\left(s_2, s_1\right):\left(s_1=x_1 \wedge s_2=x_2\right)$ and assuming that $\operatorname{gcd}\left(s_1, 10\right)=1$, it is not hard to verify that a sufficient condition which guarantees that the $2$-adic valuation rules (by their own) hold is that $\left|s_l-\alpha_{x_2 x_1}[l]\right| \notin\{0,2,4,6,8\}$, whereas (symmetrically) the stated $5$-adic valuation rules cannot be violated if $\left|s_l-\alpha_{x_2 x_1}[l]\right| \neq 5$ (in view of the fact that any issue arising from the theoretical collision $\left|s_l-\alpha_{x_2 x_1}[l]\right|=0 \Rightarrow 5\hspace{1mm}\mid \left|s_l-\alpha_{x_2 x_1}[l] \right|$ is prevented by Definition \ref{def2.3}, which implies $s_l \neq \alpha_{x_2 x_1}[l]$).

\sloppy Since the collision subtended by $2 \hspace{1mm}\mid \left|s_l-\alpha_{x_2 x_1}[l] \right|$ and the other one arising from $\left|s_l-\alpha_{x_2 x_1}[l]\right|=5$ cannot occur at the same time for any given choice of $\left(a\left(s_1\right), \alpha_{x_2 s_1}\right)$ (i.e., if $\left|s_l-\alpha_{x_2 x_1}[l]\right|=5$, then $a:\left|s_l-\alpha_{x_2 x_1}[l]\right| \notin\{0,2,4,6,8\}$, and vice versa), it is possible to opportunistically combine the $5$-adic and the $2$-adic valuation rules, in order to map the constant congruence speed of any selected $a \in \mathbb{N}-\{0,1\}$. Thus, if $a: \operatorname{gcd}(a, 10)=1$, then the constant congruence speed of every base which is greater than $10$ is described by (\ref{eq15}), and this concludes the proof of Corollary \ref{cor2.3}.
\end{proof} 
A major result of the present paper is that, by combining the statement of Theorem \ref{t2.1} and (\ref{eq15}), we are finally able to provide an explicit, unique and compact, formula that returns the exact value of the constant congruence speed of every given tetration base $a \in \mathbb{N}_0$. In order to achieve this goal, let $n, l, \alpha_{x_2 x_1}$, and the elements of the set $\left\{s_1, \ldots, s_n\right\}$ be defined as in Definition \ref{def2.3} (taking into account that $n=1 \Rightarrow s_2=0=s_l$ and that if $\nexists\left(x_2, x_1\right):\left(x_2=s_2 \wedge x_1=s_1\right)$, then $V(a)=1$ for any $a$ which does not belong to the congruence class $0$ modulo $10$).

In the light of the above, the direct map of $V(a)$ is given by (\ref{eq16}),
\begin{equation}\label{eq16}
\footnotesize{V(a)=\left\{\begin{aligned}
0 & \text { if } a: a \in\{0,1\} \\
v_5\left(a^2+1\right) & \text { if } a: a \equiv\{2, 8\}\hspace{-2.5mm}\pmod{10} \\
v_5(a+1) & \text { if } a: a \equiv 4\hspace{-2.5mm}\pmod{10} \\
v_5(a-1) & \text { if } a: a \equiv 6\hspace{-2.5mm}\pmod{10} \\
v_2(a-1) & \text { if } a: a \equiv 5\hspace{-2.5mm}\pmod{20} \\
v_2(a+1) & \text { if } a: a \equiv 15\hspace{-2.5mm}\pmod{20} \\
v_5(a-1) & \text { if } a: a \equiv 1\hspace{-2.5mm}\pmod{20} \wedge s_l \neq 5 \wedge a \neq 1 \\
v_2(a-1) & \text { if } a: a \equiv 1\hspace{-2.5mm}\pmod{20} \wedge s_l=5 \\
v_5(a-1) & \text { if } a: a \equiv 11\hspace{-2.5mm}\pmod{20} \wedge\left|s_l-\alpha_{51}[l]\right| \neq 5 \\
v_2(a+1) & \text { if } a: a \equiv 11\hspace{-2.5mm}\pmod{20} \wedge\left|s_l-\alpha_{51}[l]\right|=5 \\
v_5\left(a^2+1\right) & \text { if } a: a \equiv 3\hspace{-2.5mm}\pmod{20} \wedge\left|s_l-\alpha_{43}[l]\right| \neq 5 \\
v_2(a+1) & \text { if } a: a \equiv 3\hspace{-2.5mm}\pmod{20} \wedge\left|s_l-\alpha_{43}[l]\right|=5 \\
v_5\left(a^2+1\right) & \text { if } a: a \equiv 13\hspace{-2.5mm}\pmod{20} \wedge\left|s_l-\alpha_{93}[l]\right| \neq 5 \\
v_2(a-1) & \text { if } a: a \equiv 13\hspace{-2.5mm}\pmod{20} \wedge\left|s_l-\alpha_{93}[l]\right|=5 \\
v_5(a^2+1) & \text { if } a: a \equiv 7\hspace{-2.5mm}\pmod{20} \wedge\left|s_l-\alpha_{07}[l]\right| \neq 5 \\
v_2(a+1) & \text { if } a: a \equiv 7\hspace{-2.5mm}\pmod{20} \wedge\left|s_l-\alpha_{07}[l]\right|=5 \\
v_5\left(a^2+1\right) & \text { if } a: a \equiv 17\hspace{-2.5mm}\pmod{20} \wedge\left|s_l-\alpha_{57}[l]\right| \neq 5 \\
v_2(a-1) & \text { if } a: a \equiv 17\hspace{-2.5mm}\pmod{20} \wedge\left|s_l-\alpha_{57}[l]\right|=5 \\
v_5(a+1) & \text { if } a: a \equiv 9\hspace{-2.5mm}\pmod{20} \wedge\left|s_l-\alpha_{49}[l]\right| \neq 5 \\
v_2(a-1) & \text { if } a: a \equiv 9\hspace{-2.5mm}\pmod{20} \wedge\left|s_l-\alpha_{49}[l]\right|=5 \\
v_5(a+1) & \text { if } a: a \equiv 19\hspace{-2.5mm}\pmod{20} \wedge s_l \neq 4 \\
v_2(a+1) & \text { if } a: a \equiv 19\hspace{-2.5mm}\pmod{20} \wedge s_l=4 \\
\emptyset & \text { iff } a: a \equiv 0\hspace{-2.5mm}\pmod{10} \wedge a \neq 0
\end{aligned}\right.\hspace{2mm}.}
\end{equation}
Given the standard system of counting, radix-$(2 \cdot 5)$ (since most human beings are born with two hands and five fingers on each hand), we have shown which is the law at the bottom of the congruence speed constancy, the special property of tetration that describes the asymptotic behavior of $\# S_c(a, b)$.

Now, it is not hard to figure out how our result can be naturally extended to many different numeral systems, radix-$g$. Thus, if we simply assume that $g$ is a “valid square-free base” (see \cite{13}, Definition, p. 3, and also the Remarks, p. 4) satisfying the condition stated in Proposition 1 of Reference \cite{3}, then $V(a, b)=V(a)$ must necessarily hold for sufficiently large $b:=b(a)$, as long as $a$ is a nontrivial base such that $a \not \equiv 0\pmod {g}$.

\section{Some useful properties of the congruence speed} \label{sec:3}
The regularity features of the congruence speed \cite{11,12} can be very useful when performing peculiar mental calculations, finding also the precise value of $\# S(a, b)$ by Equation (\ref{eq1}).

\sloppy We start by saying that, for any $a: a \not \equiv 0\hspace{1mm}\pmod{10} \wedge a \neq 1, V(a, 1) \leq V(a, 2)$ always holds, so let $a$ be such that $V(a, 2)=0$ (i.e., assuming $a>1, V(a, 2)=0 \Leftrightarrow$ \linebreak
$a=((20 \cdot k+2) \vee(20 \cdot k+18)), \forall k \in \mathbb{N}_0$). Thus,
\begin{equation}\label{eq17}
  V(a, b) \geq V(a, b+1), \hspace{1mm} \forall b \geq 3 . 
\end{equation}
If $a: a \not \equiv\{0,2,10,18\}\hspace{-1mm}\pmod{20} \wedge a \neq 5$ (i.e., $a \not \equiv\{0,2,10,18\}\hspace{-1mm}\pmod{20} \Rightarrow V(a, 2) \neq 0$), then
\begin{equation}\label{eq18}
    V(a, b) \geq V(a, b+1), \hspace{1mm} \forall b \geq 2 .
\end{equation}

A general rule which is very easy to keep in mind is that $V(a, 1)+V(a, 2) \leq 3 \cdot V(a) \leq 3 \cdot \tilde{v}(a)$, with the unique exception represented by the very special base $a=1$ (since $V(1)=0$, whereas $V(1,1)>0$). Furthermore, for any $k \in \mathbb{N}_0$, let us underline that $V(a, b)=0$ if and only if $b=1$ and $a \equiv\{2,3,7,12,4,14,8,18\}\pmod{20} \vee a=0$, or if $b=2$ and $a \equiv\{2,18\}\pmod{20} \vee a=1 \vee a=0$, or if $b \geq 2$ and $a=1 \vee a=0$ (see Equations (\ref{eq4})-(\ref{eq6}), (\ref{eq11})).

Moreover, for any $a: a \not\equiv 0\pmod{10} \wedge a \neq 1$, the periodicity properties of $V(a)$ (see Equation (\ref{eq3})) let us immediately detect whether $V(a)$ is greater than $1$ or not, by simply checking the congruence $a \equiv\{2,3,4,6,8,9,11,12,13,14,16,17,19,21,22,23\} \pmod{25}$ \cite{12}; if so, $V(a)=1$, and $V(a) \geq 2$ otherwise. We can go even further and try to memorize the next set of $900$ values, $1 \leq a \not\equiv 0 \pmod{10}<1000$, in order to answer in less than one second (without writing or calculating anything) whether $V(a)=0, V(a)=1, V(a)=2$, or even $V(a) \geq 3$ (see \cite{11}, p. 252). Thus, knowing that $V(1)=0$ by definition, $\forall a \in \mathbb{N}-\{1\}: a \not\equiv 0\pmod{10}$, we have
\begin{equation}\label{eq19}
\left\{\begin{array}{c}
V(a)=1 \Leftrightarrow a\hspace{-3mm}\pmod{25} \in \mathcal{C}^C \\
\text { where } \mathcal{C}^C:=\{2,3,4,6,8,9,11,12,13,14,16,17,19,21,22,23\} \hspace{0.3mm}; \vspace{2mm} \\
V(a)=2 \Leftrightarrow a\hspace{-3mm}\pmod{40} \in\{5,35\} \hspace{0.5mm}\vee \\
\left(a \hspace{-3mm}\pmod{25} \in\{1, 7, 18, 24\} \wedge a \hspace{-3mm}\pmod{1000} \notin \mathcal{Q}^C\right); \vspace{2mm}\\
V(a) \geq 3 \Leftrightarrow a\hspace{-3mm}\pmod{40} \in\{15,25\} \vee a\hspace{-3mm}\pmod{1000} \in \mathcal{Q}^C \\
\text { where } \mathcal{Q}^C:=\left\{\begin{array}{c}
1,57,68,124,126,182,193,249,318,374,376,432,568, \\
624,626,682,751,807,818,874,876,932,943,999
\end{array}\right\} .
\end{array}\right.
\end{equation}

We can also take $\# S_c(a, b)$ and check the stable digits ratio of any integer tetration whose base is not congruent to $0$ modulo $10$. For any given ${ }^b a$, the stable digits ratio of is
\begin{equation}\label{eq20}
R(a, b):=\frac{\# S_c(a, b)}{\lceil\log _{10}({^b}a)\rceil},
\end{equation}
where the ceiling $\lceil q\rceil$ denotes the function which takes the rational number $q$ as input and returns as output the least integer greater than or equal to $q$.

In conclusion, given any tetration base $a: a \not \equiv 0\pmod{10} \wedge a \neq 1$, if we choose beforehand the desired number of stable digits (let us indicate it by $\# T(a) \in \mathbb{N}_0$ ) of $^b a$, we will easily compute which is the smallest hyperexponent
$$
\overline{\bar{b}}(a):=\min_{b}\left\{b \in \mathbb{N}-\{0\}: \sum_{i=1}^b V(a, i) \geq \# T(a)\right\}
$$
such that $^{\overline{\bar{b}}} a$ originates at least $\#T(a)$ stable digits (see \cite{10}, pp. 13-14).

Thus, ${ }^b a \equiv \hspace{0.3mm} ^{\overline{\bar{b}}} $$a$$\pmod {10^{\# T(a)}}$ for any $b(a) \geq \overline{\bar{b}}(a)$, and $\sum_{i=1}^b V(a, i)$ can be simplified using the relations shown in the present paper, e.g., by Equation (\ref{eq5}), for any $k \in \mathbb{N}_0$,
$$
\begin{aligned}
a=10 \cdot k+4 & \Rightarrow \sum_{i=1}^b V(a, i)=(b-1) \cdot v_5(a+1) \\
& \Rightarrow \overline{\bar{b}}(a)=\min _b\left\{b \in \mathbb{N}-\{0\}: b \geq\Bigl\lceil\frac{\# T(a)}{v_5(a+1)}\Bigr\rceil+1\right\} .
\end{aligned}
$$
\section{Conclusion} \label{sec:Concl}

The number of stable digits of every integer tetration \(_{}^{b}a\) such that \( a\) is not a multiple of \( 10\) is strongly related to the constant congruence speed of the base, and \(\overline{b}(a)\leq\tilde{v}(a)+2\) is a sufficient condition to guarantee the constancy of the congruence speed of \( a\) for any hyperexponent at or above \(\overline{b}(a)\) so that \(V\hspace{-0.7mm}\left(a, \overline{b}(a)+k\right) = V(a)\) for any\textit{ }\( k\in\mathbb{N}_{0}\). For this purpose, Theorem \ref{t2.1} provides an easy way to calculate the constant congruence speed of any \( a\geq 2\) that is not a multiple of \( 10\), while (\ref{eq16}) let us see in details the intrinsic structure of \( V(a)\). Finally, by combining the \( V(a)\) inverse map, 
shown in Reference  \cite{12},with a compact set of equations which allows an accurate calculation of \( \# S(a, b)\), we are starting to see some symmetrical harmony in the fascinating, chaotic, behavior of hyper-$4$.

\makeatletter
\renewcommand{\@biblabel}[1]{[#1]\hfill}
\makeatother

\end{document}